\documentclass[preprint,12pt]{elsarticle}
\usepackage{amsthm,amsfonts,amssymb,amscd,amsmath,enumerate,verbatim,calc,graphicx,geometry}
\usepackage[all]{xy}
\newtheorem{theorem}{Theorem}[section]
\newtheorem{lemma}[theorem]{Lemma}
\newtheorem{proposition}[theorem]{Proposition}
\newtheorem{corollary}[theorem]{Corollary}
\theoremstyle{definition}
\theoremstyle{definitions}
\newtheorem{definition}[theorem]{Definition}

\newtheorem{remark}[theorem]{Remark}
\newtheorem{example}[theorem]{Example}
\theoremstyle{notations}

\theoremstyle{remarks}

\journal{ }

\begin{document}

\begin{frontmatter}



\title{On Subgroups of Topologized Fundamental Groups and Generalized Coverings}


\author[]{Mehdi~Abdullahi Rashid}
\ead{mbinev@stu.um.ac.ir}
\author[]{Behrooz~Mashayekhy\corref{cor1}}
\ead{bmashf@um.ac.ir}
\author[]{Hamid~Torabi}
\ead{h.torabi@ferdowsi.um.ac.ir}
\author[]{Seyyed Zeynal~Pashaei}
\ead{pashaei.seyyedzeynal@stu.um.ac.ir}
\address{Department of Pure Mathematics, Center of Excellence in Analysis on Algebraic Structures, Ferdowsi University of Mashhad,\\
P.O.Box 1159-91775, Mashhad, Iran.}
\cortext[cor1]{Corresponding author}
\begin{abstract}
In this paper,  we are interested in study subgroups of topologized fundamental groups and their influences on generalized covering maps.
More precisely, we find some relationships between generalized covering subgroups and the other famous subgroups of the fundamental group equipped with
the compact-open topology and the whisker topology. Moreover, we present some conditions under which generalized coverings, semicoverings and coverings
are equal.
\end{abstract}


\end{frontmatter}


\section{Introduction and motivation}

We recall that a continuous map $p:\tilde{X}\rightarrow X$ is a covering map if every point of $X$ has an open neighborhood which is evenly covered by $p$. It is well-known that the induced
homomorphism $p_*:\pi_1(\tilde{X},\tilde{x})\rightarrow \pi_1(X,x)$ is a monomorphism and so $\pi_1(\tilde{X},\tilde{x})\cong p_*\pi_1(\tilde{X},\tilde{x})$ is a subgroup of $\pi_1(X,x)$. Based on some recent works of \cite{11,15,17,18,23} there is a chain of some effective subgroups of the fundamental group $\pi_1(X,x)$ as follows:
$$\{e\}\leq \pi^s_1(X,x)\leq \pi^{sg}_1(X,x)\leq \overline{{\ \pi }^{sg}_1(X,x_0)}\leq {\widetilde\pi}^{sp}_1(X,x)$$ $$\leq \pi^{sp}_1(X,x)\leq p_*\pi_1(\tilde{X},\tilde{x})\leq \pi_1(X,x),\ \ \ \ \ \ \  (*)$$
where ${\pi }^s_1(X,x)$ is the subgroup of all small loops at $x$ \cite{23}, ${\pi }^{sg}_1(X,x)$ is the subgroup of all small generated loops, i.e, the subgroup generated by the set $\{[\alpha *\beta *{\alpha }^{-1}]| [\beta]\in {\pi }^s_1(X,\alpha(1)),\alpha \in P(X,x)\}$ where $\ P(X,x)$ is the space of all paths with initial point $x$ \cite{23}, ${\pi }^{sp}_1(X,x)$ is the Spanier group of $X$, the intersection of the Spanier subgroups relative to open covers of $X$ \cite[Definition 2.3]{24}, and ${\widetilde{\pi }}^{sp}_1(X,x)$ is the path Spanier group, i.e, the intersection of all path Spanier subgroups ${\widetilde{\pi }}_1(V,x)$ where $V$ is a path open cover of $X$ \cite[Section 3]{16}.

Some people extended the notion of covering maps and introduced rigid coverings \cite{2}, semicoverings \cite{5}, generalized coverings \cite{3,11}. These generalizations focus on keeping some properties of covering maps and eliminating the evenly covered property. Brazas \cite{5} defined semicoverings by removing the evenly covered property and keeping local homeomorphism and the unique path and homotopy lifting properties. For generalized coverings, the local homeomorphism is replaced with the unique lifting property (see \cite{3,11}). In each case, one of the interesting problems is to classify subgroups of the fundamental group with respect to coverings, semicoverings and generalized coverings. A subgroup $H$ of the fundamental group ${\pi }_1(X,x)$ is called covering, semicovering, generalized covering subgroup if there is a covering, semicovering, generalized covering map $p:(\tilde{X},\tilde{x})\rightarrow (X,x)$ such that ${H=p_*\pi }_1(\tilde{X},\tilde{x})$, respectively.

Brazas \cite[Theorem 2.36]{3}  showed that the intersection of any collection of generalized covering subgroups of ${\pi }_1(X,x)$ is also a generalized covering subgroup. We give another proof for the above fact in Corollary 2.11 which is simpler than that of Brazas and denote the intersection of all generalized covering subgroups of ${\pi }_1(X,x)$ by ${\pi }^{gc}_1(X,x)$. In Section 2, we find the location of the subgroup ${\pi }^{gc}_1(X,x)$ in the chain $(*)$. In fact, we prove that ${\pi }^{sg}_1(X,x)\leq {\pi }^{gc}_1(X,x)\leq \overline{{\pi }^{sg}_1(X,x)}$ for a locally path connected space $X$. Note that $\overline{{\pi }^{sg}_1(X,x)}$ is the topological closure of the subgroup ${\pi }^{gc}_1(X,x)$ in the quasitopological fundamental group ${\pi }^{qtop}_1(X,x)$. By the quasitopological fundamental group ${\pi }^{qtop}_1(X,x)$ we mean the fundamental group ${\pi }_1(X,x)$ equipped with the quotient topology induced by the compact-open topology (see \cite{2,Br4}). Moreover, we give some examples in which ${\pi }^{sg}_1(X,x)\neq {\pi }^{gc}_1(X,x)$ and ${\pi }^{gc}_1(X,x)\neq {\pi }^{sp}_1(X,x)$.

It seems interesting to find equivalent conditions for a subgroup $H$ of the fundamental group ${\pi }_1(X,x)$ to be a covering, semicovering or generalized covering subgroup. Based on some classical results of Spanier \cite{21} every covering subgroup contains the Spanier group ${\pi }^{sp}_1(X,x)$. Hence for every subgroup $H\leq {\pi }_1(X,x)$ if ${\pi }^{sp}_1(X,x)\cap H\neq {\pi }^{sp}_1(X,x)$, then $H$ can not be a covering subgroup. If ${\pi }^{sp}_1(X,x)\cap H= {\pi }^{sp}_1(X,x)$, then by some results in \cite{24,15} $H$ is a covering subgroup if and only if $H$ contains an open normal subgroup of ${\pi }^{qtop}_1(X,x)$. For semicovering subgroups, based on some results in \cite{6,16}, one can state a similar result, i.e, if ${\widetilde{\pi }}^{sp}_1(X,x)\cap H\neq {\widetilde{\pi }}^{sp}_1(X,x)$, then $H$ can not be a semicovering subgroup. For locally path connected spaces, if ${\widetilde{\pi }}^{sp}_1(X,x)\cap H={\widetilde{\pi }}^{sp}_1(X,x)$, then $H$ is a semicovering subgroup if and only if $H$ is an open subgroup of ${\pi }^{qtop}_1(X,x)$. In Section 2, we intend to give a similar result for generalized covering subgroups. In fact, we show that if ${\pi }^{gc}_1(X,x)\cap H\neq {\pi }^{gc}_1(X,x)$, then $H$ is not a generalized covering subgroup. If ${\pi }^{gc}_1(X,x)\cap H={\pi }^{gc}_1(X,x)$, then $H$ is a generalized covering subgroup if and only if $H={(p_H)}_*{\pi }_1({\tilde{X}}_H,{\tilde{e}}_H)$, where $p_H:{\tilde{X}}_H\rightarrow X$ is the end point projection (see Definition 2.7 for the definition of $p_H:{\tilde{X}}_H\rightarrow X$). Note that Brazas \cite[Lemma 5.9]{3} gave two equivalent conditions for the equality $H={(p_H)}_*{\pi }_1({\tilde{X}}_H,{\tilde{e}}_H)$.

It is easy to show that the class of all coverings, semicoverings, generalized coverings on $X$ forms a category denoted by $\mathrm{Cov}(X)$, $\mathrm{SCov}(X)$, $\mathrm{GCov}(X)$, respectively. By definition, $\mathrm{Cov}(X)$ is a subcategory of $\mathrm{SCov}(X)$. Brazas \cite{3,6} showed that $\mathrm{Cov}(X)=\mathrm{SCov}(X)=\mathrm{GCov}(X)$ for locally path connected, semi locally simply connected space $X$. Torabi et al. \cite{19} showed that the first equality can be extended for locally path connected, semilocally small generated spaces, i.e, $\mathrm{Cov}(X)=\mathrm{SCov}(X)$. In Section 2, we show that $\mathrm{SCov}(X)=\mathrm{GCov}(X)$ for locally path connected, semilocally small generatedted space $X$.

Spanier \cite[Theorem 13 page 82]{21} introduced a different topology on the fundamental group which has been called the whisker topology by Brodskiy et al. \cite{8} and denoted by $\pi_{1}^{wh}(X,x_{0})$. They showed that $\pi_{1}^{wh}(X,x_{0})$ is a topological group when the inverse map is continuous \cite[Proposition 4.20]{8}. Although $\pi_{1}^{wh}(X,x_{0})$ is not a quasitopological group, in general, we show that it is a homogenous space (see Proposition 3.2). In Section 3, after recalling the whisker topology and presenting some of its properties, we intend to describe its influence on the notion of generalized covering subgroups. Trying to classify the generalized covering subgroups of the fundamental group, we give an example to show that the whisker topology is not suitable for the subject. Moreover, we find some properties of the whisker topology on the qoutient space $\pi_{1}^{wh}(X,x_{0})/H$, where $H$ is a subgroup of $\pi_{1}(X,x_{0})$.

In Section 4, we introduce some topological properties in order to preserve the categorical equality between $\mathrm{Cov}(X)$, $\mathrm{SCov}(X)$ and $\mathrm{GCov}(X)$. More precisely, we introduce the notions {\it semilocally path $H$-connectedness} and {\it semilocally $H$-connectedness} and give their relations to open subgroups of ${\pi }^{qtop}_1(X,x)$ and $\pi_{1}^{wh}(X,x_{0})$ in order to show that $\mathrm{SCov}(X)=\mathrm{GCov}(X)$ if and only if $X$ is a semilocally path $H$-connected, and $\mathrm{GCov}(X)=\mathrm{Cov}(X)$ if and only if $X$ is a semilocally $H$-connected when $H=\pi_{1}^{gc}(X,x_{0})$.

Throughout the paper, the topological space $X$ is assumed to be connected and locally path connected.

\section{Generalized Covering Subgroups of the Quasitopological Fundamental Group}

Universal generalized covering maps was introduced by Fischer and Zastrow \cite{11} and extended to generalized covering maps by Brazas \cite{3}. The definition is based on removing the evenly covered property and keeping the unique lifting property from classical covering maps.
\begin{definition} A continuous map $p:(\tilde{X},{\tilde{x}}_0)\rightarrow (X,x_0)$ has the unique lifting property (UL for abbreviation) if for every connected, locally path connected space$\ (Y,y_0)$ and every continuous map $f:(Y,y_0)\rightarrow (X,x_0)$ with
$f_*{\pi }_1(Y,y_0)\subseteq p_*{\pi }_1(\tilde{X},{\tilde{x}}_0)$ for ${\tilde{x}}_0\in p^{-1}(x_0)$, there exists a unique continuous map $\tilde{f}:(Y,y_0)\rightarrow (\tilde{X},{\tilde{x}}_0)$ with $p\circ\tilde{f}=f$. If $\tilde{X}$ is a connected, locally path connected space and $p:\tilde{X}\rightarrow X$ has UL property, then $p$ and $\tilde{X}$ are called a generalized covering map and a generalized covering space for $X$, respectively.
\end{definition}

\begin{definition}
A subgroup ${H\le \pi }_1(X,x_0)$ is called a generalized covering subgroup of ${\pi }_1(X,x_0)$ if there is a generalized covering $p:(\tilde{X},{\tilde{x}}_0)\rightarrow (X,x_0)$ such that ${H=p_*\pi }_1(\tilde{X},{\tilde{x}}_0)$ (see \cite {3}). Brazas \cite [Theorem 2.36]{3} showed that the intersection of any collection of generalized covering subgroups of ${\pi }_1(X,x_0)$ is also a generalized covering subgroup. We denote the intersection of all generalized covering subgroups of ${\pi }_1(X,x_0)$ by \textit{$\pi_{1}^{gc}(X,x_0)$ }.
\end{definition}

In the following proposition, we show that ${\pi }^{gc}_1(X,x_0)$ contains $ {\pi }^{sg}_1(X,x_0) $ as a subgroup of ${\pi }_1(X,x_0)$.
\begin{proposition}
For a pointed topological space $(X,x_0)$, we have
$${ {\pi }^{sg}_1(X,x_0)\le \pi }^{gc}_1(X,x_0).$$
\end{proposition}

\begin{proof} Let $g$ be any generator of $ {\pi }^{sg}_1(X,x_0) $. Then $g=[\alpha *\omega *{\alpha }^{-1}]$, where $\alpha $ is a path beginning at $x_0$ and $\omega $ is a small loop at $\alpha(1)$. Suppose $g\notin H$ for a generalized covering subgroup $H$ of ${\pi }_1(X,x_0)$. Since $H$ is a generalized covering subgroup of ${\pi }_1(X,x_0)$, by \cite [Proposition 6.4] {11} and \cite[Lemma 5.9]{3}, $X$ is homotopically Hausdorff relative to $H$ and so there exists an open neighborhood $U_g$ of $\alpha(1)$ such that there is no loop $\gamma :(I,\dot{I})\rightarrow (U_g,\alpha (1))$ with $[\alpha *\gamma *{\alpha }^{-1}]\in Hg$. Since $\omega $ is a small loop at $\alpha(1)$, there exists ${\omega }^{'}:(I,\dot{I})\rightarrow (U_g,\alpha (1))$ such that $[\omega ]=[{\omega }']$. Thus $g=[\alpha *\omega *{\alpha }^{-1}]=[\alpha *{\omega }^{'}\ast{\alpha }^{-1}]$ which implies that $[\alpha *{\omega }^{'}*{\alpha }^{-1}]\in Hg$. This is a contradiction to $X$ being homotopically Hausdorff relative to $H$. Hence ${\pi }^{sg}_1(X,x_0)\le H$ and so ${\pi}^{sg}_1(X,x_0)\le {\pi}^{gc}_1(X,x_0)$.
\end{proof}

In the case of locally path connected, semilocally simply connected spaces, equality of subgroups in the chain $ (*) $ holds since one can conclude that $ {\pi }^{sp}_1(X,x_0)=1 $. This result comes from the discreteness of ${\pi }^{qtop}_1(X,x_0)$ which supports the existence of the classical universal covering space. Note that for $ {\pi }^{sp}_1(X,x_0)=1 $, the semilocally simply connectedness is not a necessary condition. For instance, Fischer and Zastrow \cite [Example 4.15] {11} showed that the Spanier group of the Hawaiian earring, HE, is trivial which implies the existence of the generalized universal covering space for HE, i.e, $ {\pi }^{gc}_1(HE,x_0)=1 $, where $ x_0 $ is the wedge point. Using this fact and the chain $ (*) $ it is easy to see that HE is a homotopically Hausdorff space and $ {\pi }^{sg}_1(HE,x_0)={\pi }^{gc}_1(HE,x_0)=1 $.
\begin{example}
To present an example to show that $ {\pi }^{sg}_1(X,x_0) $ may be a proper subgroup of $ {\pi }^{gc}_1(X,x_0) $, consider the space $ RX $  described in \cite[Definition 7]{22}. By \cite [Theorem 16]{22}, $ RX $ is a metric, path connected, locally path connected and homotopically Hausdorff space and so by \cite{24,23} $ {\pi }^{sg}_1(X,x_0)=1 $, which does not admit a generalized universal covering space, i.e, $ {\pi }^{gc}_1(X,x_0)\neq1 $.
\end{example}

In the following theorem, using the previous proposition, we find the location of the subgroup ${\pi }^{gc}_1(X,x_0)$ in the chain  $ (*) $ of some interesting subgroups of ${\pi }_1(X,x_0)$.
\begin{theorem}
If $(X,x_0)$ is a locally path connected space, then there exsits the following chain of subgroups of ${\pi }_1(X,x_0)$.
$$\{e\}\leq {\pi }^s_1(X,x_0)\leq {\pi }^{sg}_1(X,x_0)\le \pi ^{gc}_1(X,x_0)\le \overline{{\ \pi }^{sg}_1(X,x_0)}$$ $$\leq {\widetilde{\pi }}^{sp}_1(X,x_0)\leq {\pi }^{sp}_1(X,x_0)\leq {\pi }_1(X,x_0). $$
\end{theorem}

\begin{proof} Recall from \cite[Theorem 2.2]{19} that $\overline{\{e\}}= \overline{{\ \pi }^{sg}_1(X,x_0)}$. Using Proposition 2.3, it is enough to show that $ \pi ^{gc}_1(X,x_0)\le\overline{\{e\}} $. Since $ X $ is locally path connected and $ \overline{\{e\}} $ is a closed subgroup of ${\ \pi }^{qtop}_1(X,x_0)$, the result holds by \cite[Theorem 11]{4}.
\end{proof}

\begin{example}
 As an example of a locally path connected space which supports the inequality $ {\pi }^{gc}_1(X,x_0)\neq{\pi }^{sp}_1(X,x_0) $, consider the space $ A $ in \cite [Section 3]{9}. It consists of a rotated topologists' sine curve about its limiting arc (see \cite [Fig. 1]{9}), the "central axis", where this surface tends to, and a system of horizontal arc is attached to them so that they become countable dense of radial cross sections to make the space locally path connected at the central arc. Let $ a $ be any point of central axis. Fischer et al. \cite[Proposition 3.2]{24} showed that $ {\pi }^{sp}_1(A,a)\neq1 $ and $ A $ is a homotopically path Hausdorff space. Also, they proved that every homotopically path Hausdorff space has the  generalized universal covering space \cite[Proof of Theorem 2.9]{24}. It means that $ {\pi }^{gc}_1(A,a)=1 $.
\end{example}

In order to find equivalent conditions for a subgroup $H$ of the fundamental group ${\pi }_1(X,x_0)$ to be a generalized covering subgroup we need following concepts.
\begin{definition}
Let $H$ be a subgroup of ${\pi }_1(X,x_0)$ and $P(X,x_0)$ be the path space in $X$ beginning at $x_0$. Consider an equivalence relation $\sim_H$ on $P(X,x_0)$ as follows.

${\alpha }_1\sim_H {\alpha }_2$ if and only if ${\alpha }_1(1)={\alpha }_2(1)$ and $[{\alpha }_1*{{\alpha }_2}^{-1}]\in H$.
The equivalence class of $\alpha$ denoted by ${\left\langle \alpha \right\rangle }_H$.
One can define the quotient space $\tilde{X}_H=P(X,x_0)/\sim_H$ and the map $p_H: (\tilde{X}_H,e_H)\rightarrow (X,x_0)$ defined by ${\left\langle \alpha \right\rangle }_H\rightarrow \alpha (1)$, where $e_H$ is the class of constant path at $x_0$.

If $\alpha \in P(X,x_0)$ and $U$ is an open neighborhood of $\alpha (1)$, then a continuation of $\alpha$ in $U$ is a path $\beta=\alpha *\gamma $, where $\gamma $ is a path in $U$ with $\gamma (0)=\alpha (1)$. Put $(U,{\left\langle \alpha \right\rangle }_H) =\{{\left\langle \beta \right\rangle }_H\in {\tilde{X}}_H \ | \ \mathrm{\beta\ is\ a\ continuation\ of\ \alpha\ in\ U}\}$. It is well known that the subsets $(\langle U,{\left\langle \alpha \right\rangle }_H) $ form a basis for a topology on ${\tilde{X}}_H$ for which the function $p_H:{(\tilde{X}}_H,e_H)\rightarrow (X,x_0)$ is continuous (see \cite[Page 82]{21}). Brodskiy et al. \cite{8} called this topology on ${\tilde{X}}_H$ the whisker topology.
\end{definition}

Note that the map $p_H:{(\tilde{X}}_H,e_H)\rightarrow (X,x_0)$ has the path lifting property and if $X$ is path connected, then $p_H$ is surjective (see \cite[page 83]{21}).
Brazas \cite{3} gave the following relationship between generalized covering maps and the above concepts.
\begin{lemma}(\cite[Lemma 5.10]{3}).
Suppose $\hat{p}:(\hat{X},\hat{x})\rightarrow (X,x_0)$ is a generalized covering map with ${\hat{p}}_*({\pi }_1(\hat{X},\hat{x}))=H$. Then there is a homeomorphism $h:(\hat{X},\hat{x})\rightarrow ({\tilde{X}}_H,e_H)$ such that $p_H\circ h=\hat{p}$.
\end{lemma}

A semicovering map introduced by Brazas in \cite[Definition 3.1]{5} as a local homeomorphism with continuous lifting of paths and homotopies.  Then he \cite[Definition 2.4]{6} simplified this definition by showing that the property of continuous lifting of homotopies can obtain from the continuous lifting of paths.
Further simplification could be done, such as Kowkabi et al. \cite[Theorem 3.2]{25} by showing that the continuous lifting of paths can be replaced with the unique path lifting property, i.e, the $ UL $ property with respect to paths.
Also, this result can be concluded from a new definition of semicoverings which Klevdal \cite[Definition 7]{14} presented and showed that these definitions are equivalent. In this paper, we use a continuous surjective local homeomorphism with unique path lifting property as the definition of a semicovering maps.

The following lemma which has been proved in \cite[Theorem 3.2]{13}, presented a relationship between semicoverings and generalized coverings. We give another proof using some properties of the topologized fundamental group.

\begin{lemma} Let p$:\tilde{X}\rightarrow X$ be a continuous surjective local homeomorphism with the unique path lifting property. Then $p$ has UL property.
\end{lemma}
\begin{proof}
By the above assertion, the map $p:\tilde{X}\rightarrow X$ is a semicovering map. Let $H=p_{\ast}\pi_{1}(\tilde{X},\tilde{x})\leq \pi_{1}(X,x_{0})$. Recall from \cite[Theorem 3.5]{6} that $H$ is an open subgroup of $\pi_{1}^{qtop}(X,x_{0})$. Since $\pi_{1}^{qtop}(X,x_{0})$ is a quasitopological group \cite{Br4}, $H$ is also a closed subgroup of $\pi_{1}^{qtop}(X,x_{0})$ \cite[Theorem 1.3.5]{1}. Finally, Brazas and Fabel \cite[Theorem 11]{4} showed that every closed subgroup of $\pi_{1}^{qtop}(X,x_{0})$ is a generalized covering subgroup. Hence it is done.
\end{proof}

The above lemma states that every semicovering map is a generalized covering map for connected locally path connected spaces. Clearly, the converse does not hold in general (see \cite[Example 4.15]{11}). It seems interesting to find some topological properties which guarantee the converse.

Fischer and Zastrow \cite[Lemma 2.3]{11} showed that the map $p_H:{\tilde{X}}_H\rightarrow X$ is open when $H=1$. One can prove the same result for any subgroup $H\leq {\pi }_1(X,x_0)$ with a similar proof; ``The projection map $p_H:{\tilde{X}}_H\rightarrow X$ is open if and only if $X$ is locally path connected". Thus for locally path connected spaces, the projection map $p_H$ is a local homeomorphism if and only if it is locally injective (a map $p:Y\rightarrow X$ is called locally injective if for every $y\in Y$, there exists an open neighborhood $ U $ of $ y $ such that the restriction of $ p $ on $ U $ is injective).

On the other hand, every map $p:{\tilde{X}}\rightarrow X$ with UL property is surjective when $ X $ is a path connected space. Therefore, for a connected, locally path connected  space $X$, a generalized covering map $p:\tilde{X}\rightarrow X$ is a semicovering map if and only if it is locally injective. Moreover, Fischer and Zastrow \cite[Lemma 5.5]{10} showed that  local homeomorphismness is a sufficient condition for the endpoint projection $p_H:{\tilde{X}}_H\rightarrow X$ to have UL property with respect to paths. Therefore, for a connected locally path connected space $ X $ and any subgroup $H\leq{\pi }_1(X,x_0)$, the map $p_H:{\tilde{X}}_H\rightarrow X$ is a semicovering map if and only if it is locally injective.

In the following lemma, we give necessary and sufficient conditions on a subgroup $H$ of ${\pi }_1(X,x_0)$ to be a generalized covering subgroup. Note that Fischer and Zastrow \cite[Lemma 2.8]{11} stated a result similar to the following for the trivial subgroup $H=1$. Then Brodskiy et al. \cite[Proposition 2.18]{7} extended the one side of the inclusion $(p_{H})_{\ast}\pi_{1}(\tilde{X},e_{H})\leq H$ for every generalized covering subgroup. In fact, they showed that in a path connected space $ X $, the map $p_H:{\tilde{X}}_H\rightarrow X$ has $ UPL $ property if and only if $(p_{H})_{\ast}\pi_{1}(\tilde{X},e_{H})\leq H$. The other side is obvious. Recently, Brazas \cite[Lemma 5.9]{3} gave two equivalent conditions for the equality $H={(p_H)}_*{\pi }_1({\tilde{X}}_H,{\tilde{e}}_H)$ (see also  \cite[Corollary 3.2]{26}). Now, we represent the result with another point of view.

\begin{lemma} Suppose that $H$ is a subgroup of the fundamental group ${\pi }_1\left(X,x_0\right)$. If ${\pi }^{gc}_1\left(X,x_0\right)\cap H\ne {\pi }^{gc}_1\left(X,x_0\right)$, then $H$ can not be a generalized covering subgroup. If ${\pi }^{gc}_1\left(X,x_0\right)\cap H={\pi }^{gc}_1\left(X,x_0\right)$, then $H$ is a generalized covering subgroup if and only if $H={(p_H)}_*{\pi }_1({\tilde{X}}_H,{\tilde{e}}_H)$, where $p_H:{\tilde{X}}_H\rightarrow X$ is the end point projection.
\end{lemma}
\begin{proof} The first part is obvious by the definition of $ {\pi }^{gc}_1\left(X,x_0\right) $ . For the second part, assume that ${\pi }^{gc}_1(X,x_0)\cap H={\pi }^{gc}_1\left(X,x_0\right)$. We show that if ${\left(p_H\right)}_*{\pi }_1({\tilde{X}}_H,{\tilde{e}}_H)\leq H$, then $p_H:{\tilde{X}}_H\rightarrow X$ has UL property with respect to paths. Let $ \alpha,\beta :(I,0)\rightarrow ({\tilde{X}}_H,{\tilde{e}}_H)$ be two paths with $ \alpha(0)=\beta(0)={\tilde{e}}_H $ and $ p_H \circ \alpha=p_H\circ\beta=f $. For $ s \in (0,1] $ we show that $ \alpha(s)=\beta(s) $. Let $ {\left\langle \gamma \right\rangle }_H=:\alpha(s) $ and $ {\left\langle \mu \right\rangle }_H=:\beta(s) $, where $ \gamma, \mu \in P(X,x_0) $. Since $ p_H $ has path lifting property, $ \gamma $ and $ \mu $ have standard lifts $ \tilde{\gamma}, \tilde{\mu}:(I,0)\rightarrow ({\tilde{X}}_H,\tilde{e}_H) $, respectively, i.e, $\tilde{\gamma}(t)={\left\langle \gamma_t \right\rangle }_H$ and $\tilde{\mu}(t)={\left\langle \mu_t \right\rangle }_H$ for every $ t \in I $ and so $\tilde{\gamma}(1)={\left\langle \gamma_1 \right\rangle }_H={\left\langle \gamma \right\rangle }_H=\alpha(s) $ and $ \tilde{\mu}(1)={\left\langle \mu_1 \right\rangle }_H={\left\langle \mu \right\rangle }_H=\beta(s) $. Clearly $ [ \tilde{\gamma}*\alpha^{-1}\mid_{[0,s]}*\beta\mid_{[0,s]}*\tilde{\mu}^{-1}] \in {\pi }_1({\tilde{X}}_H,{\tilde{e}}_H) $. Then by assumption $ {\left(p_H\right)}_*[ \tilde{\gamma}*\alpha^{-1}\mid_{[0,s]}*\beta\mid_{[0,s]}*\tilde{\mu}^{-1}] \in H $. Therefore $ [\gamma*\mu^{-1}]=[\gamma*f^{-1}\mid_{[0,s]}*f\mid_{[0,s]}*\mu^{-1}]=[(p_H\circ \tilde{\gamma})*(p_H\circ\alpha^{-1}\mid_{[0,s]})*(p_H\circ\beta\mid_{[0,s]})*(p_H\circ\tilde{\mu}^{-1})] \in H $ which implies that $ {\left\langle \gamma \right\rangle }_H={\left\langle \mu \right\rangle }_H $.
\end{proof}

Brazas \cite[Theorem 2.36]{3} using pullbacks showed that the intersection of any collection of generalized covering subgroups of ${\pi }_1\left(X,x_0\right)$ is also a generalized covering subgroup. In the following, using the above theorem, we give another proof which seems simpler than that of Brazas.

\begin{corollary} If $ \{H_j / j \in J\} $ is any collection of generalized covering subgroups of ${\pi }_1\left(X,x_0\right)$, then $ H=\cap_jH_j $ is a generalized covering subgroup.
\end{corollary}
\begin{proof} Using Lemma 2.10, it is enough to show that $${(p_H)}_*{\pi }_1({\tilde{X}}_H,{\tilde{e}}_H)\leq\cap_j {(p_{H_j})}_*{\pi }_1({\tilde{X}}_{H_j},{\tilde{e}}_{H_j})=H.$$

 Let $ [\alpha]=[p_H\circ{\tilde{\alpha}}]=(p_H)_*[\tilde{\alpha}] \in {(p_H)}_*{\pi }_1({\tilde{X}}_H,{\tilde{e}}_H) $, where $ \tilde{\alpha}:I\rightarrow {\tilde{X}}_H$ is a loop at $ {\tilde{e}}_H $ with $ \tilde{\alpha}(t)={\left\langle \beta_t \right\rangle }_H $. We define $ \tilde{\alpha}_j:I\rightarrow {\tilde{X}}_{H_j}$ by $ \tilde{\alpha}_j(t)={\left\langle \beta_t \right\rangle }_{H_j} $, for every $ j \in J $. It is clear that $ \tilde{\alpha}_j $ is a loop at $ {\tilde{e}}_{H_j} $ and $ p_H\circ\tilde{\alpha}=p_H\circ\tilde{\alpha}_j $. Thus $ [p_H\circ\tilde{\alpha}]=[p_H\circ\tilde{\alpha}_j]=[\alpha] $ for every $ j \in J $. Hence $ {(p_H)}_*{\pi }_1({\tilde{X}}_H,{\tilde{e}}_H)\leq H $.
\end{proof}

It is easy to check that for any subgroup $ H_0 $ of the fundamental group $ {\pi }_1(X,x_0) $, we have $ H_0\leq {(p_{H_0})}_*{\pi }_1({\tilde{X}}_{H_0}, \tilde{e}_{H_0}) $. Hence, using Lemma 2.10, if $ H_0 $ is not a generalized covering subgroup, then $ H_0\neq {(p_{H_0})}_*{\pi }_1({\tilde{X}}_{H_0},{\tilde{e}}_{H_0})$. Put $ H_{i+1}={(p_{H_i})}_*{\pi }_1({\tilde{X}}_{H_i},  \tilde{e}_{H_i}) $ for any natural number $ i\in {\mathbb N} $ and use again Lemma 2.10. This process forms an ascending chain of subgroups of the fundamental group. If $ H_i $ is a generalized covering subgroup , then by Lemma 2.10 $ H_i=H_{i+1} $ which gives the following corollary.

\begin{corollary} Let $H_0\le {\pi }_1(X,x_0)$ and define $H_{i+1}={(p_{H_i})}_*{\pi }_1({\tilde{X}}_{H_i},{\tilde{e}}_{H_i})$ for any natural number $i\in {\mathbb N}$. Then ${\{H_i\}}_{i=0}^{\infty}$ is an ascending chain of subgroups of ${\pi }_1(X,x_0)$. Moreover, $H_i$ is a generalized covering subgroup of ${\pi }_1(X,x_0)$ if and only if $H_j=H_i$ for every $ j\geq i$.
\end{corollary}

It is well-known that if $H$ is an open subgroup of a quasitopological group and $H\leq K$, then $K$ is also an open subgroup. In general, the result for $ H $ does not hold when  $K$ is an open subgroup. 
\begin{proposition} Let $K$ be an open subgroup of ${{\pi }_1}^{qtop}(X,x_0)$ and $H\le K$ such that $[K:H]<\infty $. Then $H$ is an open subgroup of ${{\pi }_1}^{qtop}(X,x_0)$ if and only if $X$ is homotopically Hausdorff relative to $H$.
\end{proposition}
\begin{proof} Since $H\le K$ and $\left[K:H\right]<\infty $, one can write $K$ as a disjoint finite union of cosets of $H$, i.e, $K=H\cup Hg_1\cup \dots \cup Hg_n$, where $g_i\notin H$ for all $1\leq j\leq n$. Since $X$ is homotopically Hausdorff relative to $H$, for each j$=1,\dots ,n$ and for every path $\alpha \in P\left(X,x_0\right)$ there is an open neighborhood $U_j$ of $\alpha \left(1\right)\in X$ such that $[\alpha *i_*{\pi }_1\left(U_j,\alpha \left(1\right)\right)*{\alpha }^{-1}]\cap Hg_j=\emptyset $. Fix$\ \alpha \in P\left(X,x_0\right)$ and put $U=\bigcap^n_{j=1}{U_j}$. Clearly $\alpha \left(1\right)\in U\ne \emptyset $ and $[\alpha *i_*{\pi }_1\left(U_j,\alpha \left(1\right)\right)*{\alpha }^{-1}]\le H$. Now one can construct a path open cover $\mathcal{U}$ of $x_0$ such that $\widetilde{\pi }(\mathcal{U},x_0)\le H$, where $\widetilde{\pi }(\mathcal{U},x_0)$ is the path Spanier subgroup of $\mathcal{U}$ (see \cite{16} for the definition). Theorem 4.1 of \cite{16} implies that $H$ is an open subgroup of ${{\pi }_1}^{qtop}(X,x_0)$. The converse is trivial.
\end{proof}

Let $X$ be a connected, locally path connected space and $H$ be a subgroup of $\pi_ {1}(X,x_0)$. Then $H$ is a semicovering subgroup if and only if it is open in ${{\pi }_1}^{qtop}(X,x_0)$ (see \cite[Theorem 3.5]{6}). Using this fact and Proposition 2.13, we have the following result for semicovering subgroups.
\begin{corollary} Let $X$ be a connected, locally path connected space and $K\leq {\pi }_1(X,x_0)$ be a semicovering subgroup. If $H\leq K$ with $[K:H]<\infty $, then $H$ is also a semicovering subgroup of ${\pi }_1(X,x_0)$.
\end{corollary}
\begin{remark}
The class of all coverings, semicoverings, generalized coverings on $X$ forms a category denoted by  $\mathrm{Cov}(X)$, $\mathrm{SCov}(X)$, $\mathrm{GCov}(X)$, respectively. A morphism between two objects $p:\tilde{X}\rightarrow X$ and $q:\tilde{Y}\rightarrow X$ is a map $h:\tilde{Y}\rightarrow \tilde{X}$ such that $p\circ h=q$. It is easy to see that $\mathrm{Cov}(X)$ is a subcategory of $\mathrm{SCov}(X)$. Moreover, Lemma 2.9 implies the following:
$$\mathrm{Cov}(X) \leq \mathrm{SCov}(X) \leq \mathrm{GCov}(X).$$
\end{remark}

Existence of universal objects in the above categories is considered by some people. Pakdaman et al. \cite[Theorem 3.12]{18} showed that the Spanier covering is the universal object in the category of  $\mathrm{Cov}(X)$ for a locally path connected space $ X $. We recall from \cite[Definition 2.4]{15} that a space $ X $ is called coverable if the universal object exists in the category $\mathrm{Cov}(X)$. Brazas  presented a similar result for the category $\mathrm{SCov}(X)$. He introduced the universal object of the category $\mathrm{SCov}(X)$ in \cite[Definition 3.1]{5} and gave properties which support the existence of it \cite[Corollary 7.21]{5}.
 In each case the universal object may not exist, in general.  However, in the category $\mathrm{GCov}(X)$ the universal object exists, in general since it is easy to see from Corollary 2.11 that $ {\pi }^{gc}_1(X,x_0) $ is a generalized covering subgroup. It means that the map $ p_H:\tilde{X}_H\rightarrow X $ is a generalized covering, when $ H={\pi }^{gc}_1(X,x_0) $. Therefore, we have the following result.
\begin{corollary}
$ \tilde{X}_{{\pi }^{gc}_1(X,x_0)} $ is the universal object in the category $\mathrm{GCov}(X)$.
\end{corollary}

Using\cite[Corollary 7.2]{5} and \cite[Example 5.4]{3} one can conclude easily that $\mathrm{Cov}(X)=\mathrm{SCov}(X)=\mathrm{GCov}(X)$ for a locally path connected, semilocally simply connected space $X$. Torabi et al. \cite{19} extended the first equality for locally path connected, semilocally small generated spaces, i.e, $\mathrm{Cov}(X)=\mathrm{SCov}(X)$. Now in the following, we extend the result to generalized coverings.

\begin{proposition} If $X$ is a connected, locally path connected and semilocally small generated space, then
$$\mathrm{Cov}(X)=\mathrm{SCov}(X)=\mathrm{GCov}(X).$$
\end{proposition}

\begin{proof} Since $X$ is semilocally small generated, there exists an open cover $U$ of $X$ such that ${\pi (U,x)\le {\pi }^{sg}_1(X,x)\le \pi }^{gc}_1\left(X,x\right)$. If $H$ is a generalized covering subgroup of ${\pi }_1\left(X,x\right)$, then ${\pi (U,x)\le \pi }^{gc}_1\left(X,x\right)\le H$ which implies that $H$ is a covering subgroup of ${\pi }_1\left(X,x\right)$. Hence $\mathrm{GCov}(X)\leq \mathrm{Cov}(X)$. The converse holds by Remark 2.15.
\end{proof}

The following corollaries seem interesting which are consequences of main results of this section.

\begin{corollary} If $f:Y\rightarrow X$ has UL property and $ X $ is semilocally small generated, then $ f $ is a local homeomorphism.
\end{corollary}
\begin{proof}
Recall from Definition 2.1 that $f:Y\rightarrow X$ is a generalized covering map. Hence Proposition 2.17 implies that $f:Y\rightarrow X$ is a semicovering map.
\end{proof}

\begin{corollary} If $H\leq{{\pi }_1}(X,x_0)$, the endpoint projection map $ p_H:\tilde{X}_H\rightarrow X $ has UL property only with respect to paths and $ X $ is semilocally small generated, then $ p_H $ is a covering map.
\end{corollary}
\begin{proof}
Note that $UL$ property of $ p_H:\tilde{X}_H\rightarrow X $ comes from $UL$ property with respect to paths. Hence the result holds by Proposition 2.17.
\end{proof}


\section{On the Whisker Topology and Generalized Covering Maps}

For any pointed topological space $(X,x_{0})$ the whisker topology on the set $\tilde{X}_{e}$ (as shown in Definition 2.7) can be inherited on $\pi_{1}(X,x_{0})$ by the bijection $f:\pi_{1}(X,x_{0})\rightarrow p^{-1}_{e}(x_{0})$ defined by $[\alpha]\rightarrow \langle\alpha\rangle_{e}$, where $e$ is the trivial subgroup of $\pi_{1}(X,x_{0})$. The fundamental group with the whisker topology is denoted by ${{\pi }_1}^{wh}(X,x_0)$. Fischer and Zastrow \cite[Lemma 2.1]{11} have shown that the whisker topology is finer than the inherited topology from the compact-open topology on ${\pi }_1(X,x_0)$ which is denoted by ${{\pi }_1}^{qtop}(X,x_0)$. In this section, we intend to present some properties of ${{\pi }_1}^{wh}(X,x_0)$ in comparison with ${{\pi }_1}^{qtop}(X,x_0)$ in order to classification of generalized covering subgroups.

At first, it seems necessary to characterize open subsets and open subgroups of ${{\pi }_1}^{wh}(X,x_0)$. Let $\left[\alpha \right]\in {\pi }_1\left(X,x_0\right)$, then for any open neighborhood $U$ of $x_0$ there is a bijection ${\varphi }_{\alpha }:i_*{\pi }_1\left(U,x_0\right)\rightarrow \left(U,\left[\alpha \right]\right)\cap {p_e}^{-1}\left(x_0\right)$ defined by ${\varphi }_{\alpha }:\left[\gamma \right]\longmapsto [\alpha *\gamma ]$.

\begin{lemma}
 The collection $ \lbrace [\alpha]i_{\ast}\pi_{1}(U,x_{0}) \ \vert \  [\alpha] \in{\pi_{1}(U,x_{0})}, \ {\mathrm where} \ U \ {\mathrm is}\ {\mathrm an} \ {\mathrm open} \ {\mathrm neighborhood} \ {\mathrm of} \ x_{0} \rbrace$ forms a basis for the whisker topology on ${\pi }_1\left(X,x_0\right)$. Moreover, these basis elements are also closed.
\end{lemma}

\begin{proof}
If $ [\beta]\in{[\alpha]i_{\ast}\pi_{1}(U,x_{0}) }$ for any $[\beta]\in{\pi_{1}(U,x_{0})}$ and some open neighborhood $U$ of $x_{0}$, then it is easy to show that  $[\beta]i_{\ast}\pi_{1}(U,x_{0})=[\alpha]i_{\ast}\pi_{1}(U,x_{0})$  and vise versa. If $[\beta ]\notin \left[\alpha \right]i_*{\pi }_1\left(U,x_0\right)$, then $\left[\beta \right]i_*{\pi }_1\left(U,x_0\right)\cap \left[\alpha \right]i_*{\pi }_1\left(U,x_0\right)=\emptyset $. Therefore, $\left[\alpha \right]i_*{\pi }_1\left(U,x_0\right)$ is closed in ${{\pi }_1}^{wh}\left(X,x_0\right)$.
\end{proof}

Despite ${{\pi }_1}^{wh}(X,x_0)$ is not a topological group in general (it is not even a quasitopological group, for example, in ${{\pi }_1}^{wh}(HE,x_0)$ right translations are not continuous), it has some properties of topological groups. The following propositions state some of them. Recall that a non-empty topological space $X$ is called a $G$-space, for a group $G$, if it is equipped with an action of $G$ on $X$. A {\textit{homogeneous space}} is a $G$-space on which $G$ acts transitively.
\begin{proposition}
 ${{\pi }_1}^{wh}(X,x_0)$ is a homogenous space.
\end{proposition}

\begin{proof}
Clearly ${{\pi }_1}^{wh}(X,x_0)$ acts on itself by translation. To show transitivity of this action, it is enough to show that the left translations are homeomorphism \cite[Corollary 1.3.6]{1}. Let ${\varphi }_{\alpha }:{{\pi }_1}^{wh}\left(X,x_0\right)\rightarrow {{\pi }_1}^{wh}\left(X,x_0\right)$ with $\left[\beta \right]\longmapsto [\alpha *\beta ]$, for any $ \alpha\in{\pi_{1}^{wh}(X,x_{0})}$, be a left translation. Clearly $\ {\varphi }_{\alpha }$ is a bijection. We show that ${\varphi }_{\alpha }$ is continuous. Let $U$ be an open neighborhood of $x_0$ in $X$ and $[\alpha *\beta ]\ i_*{\pi }_1(U,x_0)$ be an open basis neighborhood of $[\alpha *\beta ]$ in ${{\pi }_1}^{wh}(X,x_0)$. It is clear that $[\beta ]\ i_*{\pi }_1(U,x_0)$ is an open basis neighborhood of $[\beta ]$ and ${\varphi }_{\alpha }(\left[\beta \right]i_*{\pi }_1\left(U,x_0\right))\subseteq [\alpha *\beta ]\ i_*{\pi }_1(U,x_0)$.
\end{proof}

An immediate consequence of the above proposition is that every open subgroup of $\pi_{1}^{wh}(X,x_{0})$ is also closed. Recall that a topological space is called {\textit{totally separated}} if for every pair of disjoint points there exists a clopen subset which contains one of the points and doesn't contain the other. The following proposition states separation axioms for ${{\pi }_1}^{wh}\left(X,x_0\right)$.
\begin{proposition}
 For a connected locally path connected space $X$, the following statements are equivalent:
\begin{enumerate}
\item  ${{\pi }_1}^{wh}\left(X,x_0\right)$ is $T_0$.
\item  ${{\pi }_1}^{wh}\left(X,x_0\right)$ is $T_1$.
\item  ${{\pi }_1}^{wh}\left(X,x_0\right)$ is $T_2$.
\item  ${{\pi }_1}^{wh}\left(X,x_0\right)$ is $T_3$ ($T_3=regular+T_1)$.
\item  ${{\pi }_1}^s\left(X,x_0\right)=1$.
\item  ${{\pi }_1}^{wh}\left(X,x_0\right)$ is totally separated.
\end{enumerate}
Moreover, ${{\pi }_1}^{wh}\left(X,x_0\right)$ is regular, in general.
\end{proposition}
\begin{proof} $1)\Rightarrow 2)$. Let $\left[g\right],[f]\in {{\pi }_1}^{wh}(X,x_0)$ and $U$ be an open neighborhood of $x_0$ in $X$ such that $\left[g\right]\notin [f]i_*{\pi }_1(U,x_0)$. We show that $\left[f\right]\notin [g]i_*{\pi }_1(U,x_0)$. By contrary suppose $\left[f\right]\in [g]i_*{\pi }_1(U,x_0)$, then there exists a loop $\gamma $ at $x_0$ with $\gamma (I)\subseteq U$ such that $\left[f\right]=\left[g*\gamma \right]$ and so $[g]=\left[f*{\gamma }^{-1}\right]$. Thus, it shows that $\left[g\right]\in [f]i_*{\pi }_1(U,x_0)$ which is a contradiction.

$2)\Rightarrow 3)$. Let $\left[g\right],[f]\in {{\pi }_1}^{wh}(X,x_0)$ and $U,V$ be open neighborhoods of $x_0$ in $X$ such that $\left[g\right]\notin [f]i_*{\pi }_1(V,x_0)$ and $\left[f\right]\notin [g]i_*{\pi }_1(U,x_0)$. Put $W=U\cap V$ and suppose that  $\left[g\right]i_*{\pi }_1\left(W,x_0\right)\cap \left[f\right]i_*{\pi }_1\left(W,x_0\right)\ne \emptyset $. Thus, there exist $\left[{\mu }_1\right],\ [{\mu }_2]\in i_*{\pi }_1\left(W,x_0\right)$ with $\left[f*{\mu }_1\right]=[g*{\mu }_2]$. Hence $\left[f\right]=[g*{\mu }_2*{{\mu }_1}^{-1}]\in \left[g\right]i_*{\pi }_1\left(W,x_0\right)$ which contradicts to the hypothesis. Therefore, $\left[g\right]i_*{\pi }_1\left(W,x_0\right)\cap \left[f\right]i_*{\pi }_1\left(W,x_0\right)=\emptyset $.

$3)\Rightarrow 5)$. Let $1\ne [g]\in {{\pi }_1}^{wh}(X,x_0)$, then there exist open neighborhoods $U,V$ of $x_0$ in $X$ such that $[g]i_*{\pi }_1(U,x_0)\cap i_*{\pi }_1(V,x_0)=\emptyset $. Thus $[g]\notin i_*{\pi }_1(V,x_0)$. It implies that $g$ is not a small loop at $x_0$. Therefore ${{\pi }_1}^s(X,x_0)=1$.

$5)\Rightarrow 1)$. Let $[g],[f]\in {{\pi }_1}^{wh}(X,x_0)$ and $[g]\ne [f]$. Thus $[f*g^{-1}]\ne 1$. Since ${{\pi }_1}^s(X,x_0)=1$, $[f*g^{-1}]$ is not a small loop at $x_0$. Therefore, there exists an open neighborhood $V$ of $x_0$ in $X$ such that $[f*g^{-1}]\notin i_*{\pi }_1(V,x_0)$, i.e., $[f]\notin [g]i_*{\pi }_1(V,x_0)$.

$1)\leftrightarrow 6)$. By Lemma 3.1  every $[g]i_*{\pi }_1(V,x_0)$ is a clopen subset of ${{\pi }_1}^{wh}(X,x_0)$. Hence if ${{\pi }_1}^{wh}(X,x_0)$ is $T_0$, then it is totally separated. The converse holds, in general. To show that 4) is equivalent to the other statements, it is enough to show that ${{\pi }_1}^{wh}(X,x_0)$  is regular, in general. Let $A$ be a closed subset of ${{\pi }_1}^{wh}(X,x_0)$ and $[g]\notin A$. By definition of closed subsets, there exists an open neighborhood $V$ of $x_0$ in $X$ such that $[g]i_*{\pi }_1(V,x_0)\cap A=\emptyset $. Define
\[Ai_*{\pi }_1(V,x_0)=\cup \{[f]i_*{\pi }_1(V,x_0):[f]\in A\}.\]
Clearly $Ai_*{\pi }_1(V,x_0)$ is an open subset of ${{\pi }_1}^{wh}(X,x_0)$ which contains $A$. It is enough to show that $[g]i_*{\pi }_1(V,x_0)\cap (Ai_*{\pi }_1(V,x_0))=\emptyset $. Suppose that there exist $[{\mu }_1],\ [{\mu}_2]\in i_*{\pi }_1(V,x_0)$ and $[f]\in A$ such that $[g*{\mu }_1]=[f*{\mu }_2]$. Then $[f]=[g*{\mu }_1*{{\mu }_2}^{-1}]\in [g]i_*{\pi }_1(V,x_0)$ a contradiction.
\end{proof}

  It is  well-known that fibers of a covering map are homeomorphic. This fact may not be true in the case of generalized covering maps. As an example, consider the Hawaiian earring that has universal generalized covering map (see \cite[Example 4.15]{11}). Fischer and Zastrow \cite{11} showed that the fibers of this universal generalized covering map are not homeomorphic. This fact implies also from \cite[Proposition 4.21]{8} which states that ${{\pi }_1}^{wh}(X,x_0)$ is discrete if and only if $X$ is semilocally simply connected at $x_0$. We know that the fiber of $ x_0 $ of the generalized universal covering map is homeomorphic to ${{\pi }_1}^{wh}(X,x_0)$. Let $x_0 \in HE$ be the wedge point of all circles in $HE$. Clearly $HE$ is semilocally simply connected at each of its points but $x_0 \in HE$. Then ${{\pi }_1}^{wh}(HE,x_0)$ is not discrete and ${{\pi }_1}^{wh}(HE,x)$ is discrete in each other point $x \in HE$. This shows that the fibers $p^{-1}(x_0)$ and $p^{-1}(x)$ are not homeomorphic in the whisker topology. Fischer and Zastrow also showed that for generalized universal covering maps, fibers have the same cardinality. In the following proposition, we extend the result for any subgroup of ${\pi }_1(X,x_0)$.

\begin{proposition}
For a connected locally path connected space $X,$ let $H{\le \pi }_1(X,x_0)$ and $p_H:{\tilde{X}}_H\rightarrow X$ be the endpoint projection ($p_H({\langle \alpha \rangle }_H)=\alpha (1)$). Then all fibers over each points of $X$ have the same cardinality.
\end{proposition}

\begin{proof}
 Let $x,y\in X$ and define $ \psi:{(p_H)}^{-1}(x)\rightarrow {(p_H)}^{-1}(y) $ by ${\langle \alpha \rangle }_H\longmapsto {\langle \alpha *\gamma \rangle }_H $, where $\gamma :I\rightarrow X$ is a path from \textit{x} to \textit{y}. It is easy to show that $ \psi $ is a well-defined bijection.
\end{proof}

\begin{corollary} Let $H\le {\pi }_1(X,x_0)$, then $|{p_H}^{-1}(x_0)|$=$[\pi _{1}(X,x_{0}):H]$.
\end{corollary}

Virk and Zastrow \cite{22} presented an example of a space $X$ which is homotopically Hausdorff but doesn't admit a generalized universal covering. In the following theorem, we show that a sufficient condition on a homotopically Hausdorff relative to $H$ space to admit $H$ as a generalized covering subgroup is that $H$ to be a dense subgroup in ${{\pi }_1}^{wh}(X,x_0)$. We denote the UL property with respect to paths by $UPL$.

\begin{theorem} Let $X$ be homotopically Hausdorff relative to $H$ and $H$ be a dense subgroup of ${{\pi }_1}^{wh}(X,x_0)$. Then $p_H:{\tilde{X}}_H\rightarrow X$ has UPL property.
\end{theorem}

\begin{proof} Suppose $p_H:{\tilde{X}}_H\rightarrow X$ doesn't have UPL property and let $\alpha \epsilon P(X,x_0)$ be a path with a lift $\beta :I\to {\tilde{X}}_H,\ \ \ \beta (t)={\langle {\beta }_t\rangle }_H$, $\beta (0)={\tilde{x}}_H=\widetilde{\alpha }(0)$ distinct from its standard lift $\widetilde{\alpha }(t)={\langle {\alpha }_t\rangle }_H$. Then there is $t_0\in (0,1]$ such that $\beta (t_0)\ne \widetilde{\alpha }(t_0)$. Thus ${\langle {\alpha }_{t_0}\rangle }_H\ne {\langle {\beta }_{t_0}\rangle }_H$ and so $[{\alpha }_{t_0}*{{\beta }_{t_0}}^{-1}]$$\notin $$H$. Since $X$ is homotopically Hausdorff relative to $H$, there is an open neighborhood $U$ of $x_{0}$ such that $i_{\ast}\pi_{1}(U,x_{0}) \cap [\alpha_{t_{0}} \ast \beta_{t_{0}}^{-1}]H=\emptyset$ which implies $[\beta_{t_{0}} \ast \alpha_{t_{0}}^{-1}]i_{\ast}\pi_{1}(U,x_{0}) \cap H=\emptyset $. This contradicts to the density of $H$ in $\pi_{1}^{wh}(X,x_{0})$.
\end{proof}

\begin{corollary}
 Let $H$ be a dense subgroup of $\pi_{1}^{wh}(X,x_{0})$, then the map $p_{H}:\tilde{X}_{H}\rightarrow X$ is a generalized covering map (or equivalently $H$ is a generalized covering subgroup) if and only if $X$ is homotopically Hausdorff relative to $H$.
\end{corollary}

 Torabi et al. \cite[Theorem 2.2]{19} showed that the closure of the trivial subgroup of $\pi_{1}^{qtop}(X,x_{0})$ contains $\pi_{1}^{sg}(X,x_{0})$. In the whisker topology ${{\pi }_1}^{wh}(X,x_0)$, this closure is equal to  $\pi_{1}^{s}(X,x_{0})$, the collection of all small loops at $x_0$.
\begin{proposition}
The closure of the trivial element in ${{\pi }_1}^{wh}(X,x_0)$ equals ${{\pi }_1}^s(X,x_0)$.
\end{proposition}
\begin{proof} If for any $[f]\in {{\pi }_1}^{wh}(X,x_0)$ and each neighborhood $V$ of $x_0$ in $X$, $1\in [f]i_*{\pi }_1(V,x_0)$, then $[f]i_*{\pi }_1(V,x_0)=i_*{\pi }_1(V,x_0)$. Thus the closure of the trivial subgroup of the fundamental group equipped with  the whisker topology equals with the intersection of all $ i_*{\pi }_1(V,x_0) $, where $ V $ is an open neighborhood of $ x_0 $ in $ X $. Clearly, this intersection only contains small loops at $x_0$.
\end{proof}

If $H\le {{\pi }_1}^{qtop}(X,x_0)$ is a closed subgroup, then $p_H:{\tilde{X}}_H\rightarrow X$ has ${UPL}$ property \cite[Theorem 11]{4}. This result will be changed when the fundamental group equipped with the whisker topology.

\begin{proposition}
 If $X$ is a homotopically Hausdorff relative to $H$ space, then $H$ is a closed subgroup of ${{\pi }_1}^{wh}(X,x_0)$.
\end{proposition}

\begin{proof}
Let $[{\beta }^{-1}]\in{\pi }_1(X,x_0)\backslash H$ and put $\alpha $ a constant path at $x_0\in X$. Since $X$ is homotopically Hausdorff relative to $H$, there exists an open neighborhood $U$ of $x_0$ in $X$ such that for every $[\gamma]\in {\pi }_1(U,x_0)$, $[\gamma]\ \notin [{\beta }^{-1}]H $ and so $ [\beta *\gamma]\notin H $. Hence  $ [\beta]i_*{\pi }_1(U,x_0)\cap H=\emptyset$ which implies that $ H$ is a closed subgroup of ${{\pi }_1}^{wh}(X,x_0)$.
\end{proof}

As a consequence of the above proposition, we can state the following corollary using this fact that if $ p_{H}:\tilde{X}_{H}\rightarrow X$ is a generalized covering map, then $X$ is Homotopically Hausdorff relative to $H$.

\begin{corollary}
If $H\le {\pi_{1}(X,x_{0})}$ is a generalized covering subgroup, then $H$ is a closed subgroup of ${{\pi }_1}^{wh}(X,x_0)$.
\end{corollary}
\noindent \textbf{Question.}
 If $H\le {\pi }_1(X,x_0){\rm \ }$ is a closed subgroup of ${{\pi }_1}^{wh}(X,x_0)$, then when dose $p_H:{\tilde{X}}_H\rightarrow X$ have the ${UPL}$ property?

As mentioned before, in both the whisker topology and the compact-open topology on the fundamental group every open subgroups are also closed, but the converse may not hold, in general. Hence an open subgroup may illustrate some properties which does not hold for the closed one.
Therefore, a question may be raised naturally that if $H$ is an open subgroup of ${{\pi }_1}^{wh}(X,x_0)$, then whether $p_H:{\tilde{X}}_H\rightarrow X$ has the ${UPL}$ property or not? By the following example we show that there is an open (also closed) normal subgroup of the fundamental group equipped with the whisker topology that does not admit a generalized covering map.
\begin{example}
Consider the {\textit{Harmonic Archipelago }}, $HA$ (see \cite{23}). Let $a\in HA$ be the canonical based point and put $b\ne a$ be another point. Using \cite [Proposition 4.21]{8}, ${{\pi }_1}^{wh}(HA,b)$  is discrete. Thus the trivial subgroup of ${{\pi }_1}^{wh}(HA,b)$ is an open and closed subgroup. It is clearly a normal subgroup.

On the other hand, ${{\pi }_1}^{sg}(HA,b)={{\pi }_1}(HA,b)$ (see \cite[Remark 2.11]{19}). Then by Proposition 2.3, ${{\pi }_1}^{gc}(HA,b)={{\pi }_1}(HA,b)$. Hence there is only the trivial generalized covering which is also a covering. Therefore ${{\pi }_1}^{gc}(HA,b)\ne 1$ and so $(HA,b)$ has no generalized universal covering. Thus the trivial subgroup does not admit the generalized covering.
\end{example}

For every subgroup $H\le {{\pi }_1}^{wh}(X,x_0)$, there is a homeomorphism $h:{p_H}^{-1}(x_0)\rightarrow {{{\pi }_1}^{wh}(X,x_0)}/{H}$ when ${{{\pi }_1}^{wh}(X,x_0)}/{H}$ is as quotient space of ${{\pi }_1}^{wh}(X,x_0)$ with the quotient topology.
 Let $h:\ {\langle \alpha \rangle }_H\longmapsto [\alpha ]H$, then the following diagram commutes when $s$ is defined by ${{\langle \alpha \rangle }_e\longmapsto \langle \alpha \rangle }_H$. Clearly $s$ is a quotient map and $f$ is a homeomorphism. Since $q$ is a quotient map, $h$ is a homeomorphism.

\begin{displaymath}
    \xymatrix{
        p^{-1}(x_{0}) \ar[r]^f \ar[d]_s & \pi_{1}^{wh}(X,x_{0}) \ar[d]^{q} \\
        p_{H}^{-1}(x_{0}) \ar[r]_{h}       & \pi_{1}^{wh}(X,x_{0})/H.}
\end{displaymath}

Clearly, for every normal subgroup $H\unlhd {{\pi }_1}^{wh}(X,x_0)$, the fiber ${p_H}^{-1}(x_0)$ will be a group by the binary operation $\psi :{(p^{-1}_H(x_0))}^{wh}\times \ {(p^{-1}_H(x_0))}^{wh}\to \ {((p^{-1}_H(x_0))}^{wh}$ defined by $\psi({\langle \alpha \rangle}_H\ ,{\langle \beta \rangle }_H))={\langle \alpha *\beta \rangle }_H$. In general, there is a useful action for any subgroup $H\le{\pi }_1(X,x_0)$ that may help to explain the structure of $p^{-1}_H(x_0)$. The map $*_H:\ p^{-1}_H(x_0)\times{\pi }_1(X,x_0)\to p^{-1}_H(x_0)$ defined by $*_H({\langle \alpha \rangle }_H\ ,\ [\gamma ])={\langle \alpha *\gamma \rangle }_H$ is a group action with the following properties:
\begin{enumerate}
\item   If  $\tilde{x}={\langle \alpha \rangle }_H\in  p^{-1}_H (x_0)$, then  the stabilizer of  $\tilde{x}$ is $[{\alpha }^{-1}H\alpha]$.

\item   O(${\langle e_0\rangle }_H)={\langle e_0\rangle }_H*\ {\pi }_1(X,x_0)=p^{-1}_H(x_0)\ $, where O(${\langle \alpha \rangle }_H)$ is the orbit of ${\langle \alpha \rangle }_H$.
\end{enumerate}
Part 1 holds since if ${\pi }_1{(X,x_0)}_{\tilde{x}}=:\{[g]\ \in {\pi }_1(X,x_0)\mathrel{|\vphantom{[g]\ \in {\pi }_1(X,x_0) \ \tilde{x}*_H[g]={\rm \ }\tilde{x}}\kern-\nulldelimiterspace}\ \tilde{x}*_H[g]={\rm \ }\tilde{x}\}$ is the stabilizer of $\tilde{x}\in p^{-1}_H(x_0)$, then
$\pi_{1}(X,x_{0})_{\tilde{x}}=\lbrace [g]\in{\pi_{1}(X,x_{0})} \ \vert \ \langle \alpha \ast g \rangle_{H}=\langle \alpha\rangle_{H}\rbrace=\lbrace [g]\in{\pi_{1}(X,x_{0})} \ \vert \ [\alpha\ast g\ast \alpha^{-1}]\in{H}\rbrace$.
Therefore, ${\pi }_1{(X,x_0)}_{\tilde{x}}=[{\alpha }^{-1}H\alpha ]$. The second part is a trivial consequence of part 1. This action easily implies another simple proof for $|p^{-1}_H(x_0)|=[{\pi }_1(X,x_0):H]$, because $|p^{-1}_H(x_0)|=|{\rm o(}e_H{\rm )}|=[{\pi }_1(X,x_0):{\pi }_1{(X,x_0)}_{e_H}]=[{\pi }_1(X,x_0):H]$.

In the following corollary  parts 1, 2 and 3 are immediate consequences of Proposition 3.4. Part 4 is a consequences of the above argument.
\begin{corollary}
Let $p_H:{\tilde{X}}_H\ \to X$ be a generalized covering map, then
\begin{enumerate}
\item   For every pair $ \tilde{x}_{1},\tilde{x}_{0}\in{p^{-1}_H(x_0)}$, ${(p_H)}_*{\pi }_1(\tilde{X}_H, \tilde{x}_0) $ and ${(p_H)}_*{\pi }_1({\tilde{X}}_H,{\tilde{x}}_{1})$ are conjugate subgroups of  ${\pi }_1(X,x_0)$.
\item   $|p^{-1}_H(x_0)|=[{\pi }_1(X,x_0):{(p_H)}_*{\pi }_1({{X}}_H,{\langle e_0\rangle }_H)]$.
\item   If $K\le \ {\pi }_1(X,x_0)$ is conjugate to ${(p_H)}_*{\pi }_1({\tilde{X}}_H,\tilde{x})$ for $\tilde{x}=\ {\langle \alpha \rangle }_H\in{p^{-1}_H(x_0)}$, then $K={(p_H)}_*{\pi }_1({\tilde{X}}_H,\tilde{y})$ for some $\tilde{y}={\langle \alpha *\gamma \rangle }_H$, where ${(p_H)}_*{\pi }_1({\tilde{X}}_H,\tilde{x})=[\gamma ]K[{\gamma }^{-1}]$.
\item  $|O({\langle \alpha \rangle }_H)|=[{\pi }_1(X,x_0):[{\alpha }^{-1}H\alpha]]$.
\end{enumerate}
\end{corollary}

 Similar to the trivial case $H={1}$, there is a bijection $\varphi :{\langle \alpha \rangle}_H*_Hi_*{\pi }_1(U,x_0)\to (U,{\langle \alpha \rangle }_H)\cap p^{-1}_H(x_0)$ defined by $({\langle \alpha \rangle }_H*_H[\gamma])={\langle \alpha *\gamma \rangle }_H$,
for every $H\leq\pi_{1}(X,x_{0})$, any ${\langle \alpha \rangle }_H\in p^{-1}_H(x_0)$ and every open neighborhood $U$ of $x_0$, where $p_H:{\tilde{X}}_H\to X$ is the endpoint projection map. The above action extends Proposition 3.2  for every normal subgroup $H$.
\begin{proposition}
 Let $H$ be a normal subgroup of ${\pi }_1(X,x_0)$, then $p^{-1}_H(x_0)$ is a homogenous space.
\end{proposition}
\begin{proof}
 It is enough to show that the left translations are homeomorphism. Let $ {\omega }_{\alpha }:p^{-1}_H(x_0)\to p^{-1}_H(x_0) $ with $ {\langle \beta \rangle }_H\longmapsto {\langle \alpha *\beta \rangle }_H $ for any $ {\langle \alpha \rangle }_H\in p^{-1}_H(x_0)$ be a left translation. Clearly, $\ {\omega }_{\alpha }$ is a bijection. We show that ${\omega }_{\alpha }$ is continuous. Let $U$ be an open neighborhood of $x_0$ in $X$ and ${\langle \alpha *\beta \rangle }_H*_Hi_*{\pi }_1(U,x_0)$ be an open basis neighborhood of ${\langle \alpha *\beta \rangle }_H$ in $p^{-1}_H(x_0)$. It is clear that ${\langle \beta \rangle }_H*_Hi_*{\pi }_1(U,x_0)$  is an open basis neighborhood of ${\langle \beta \rangle }_H$ and ${\omega }_{\alpha }({\langle \beta \rangle }_H*_Hi_*{\pi }_1(U,x_0))\subseteq {\langle \alpha *\beta \rangle }_H*_Hi_*{\pi }_1(U,x_0)$. Hence the result holds.
\end{proof}

\section{Semilocally (Path) Connectedness with Respect to a Subgroup}

It is well-known that semilocally simply connectedness plays an important roles in the covering space theory. For example, a locally path connected space has universal covering if and only if it is semilocally simply connected. Recall that a space $X$ is called semilocally simply connected at ${x}\in{X}$ if there exists an open neighborhood $U\ of\ x$ such that all loops in $U$ at $x$ are nullhomotopic in $X$. Moreover, a space is called semilocally simply connected if it is semilocally simply connected at each of its points. In this section, we extend these concepts to any subgroup of the fundamental group and introduce their topological equivalences.

\begin{definition}
Let $X$ be a topological space and $H$ be a subgroup of $\pi_{1}(X,x_{0})$. Then we define the following notions.\\
\textbf{(i)} $X$ is called {\textit{semilocally H-connected at }} $x_{0}\in{X}$ if there exists an open neighborhood $U\ of\ x_{0}$ with $i_*{\pi }_1(U,x_{0})\le H$.\\
\textbf{(ii)} $X$ is called {\textit{semilocally path H-connected}} if for every path $\alpha$ beginning at $x_{0}$ there exists an open neighborhood $U_{\alpha }$ of  $\alpha (1)$ with $i_*{\pi }_1\left(U_{\alpha },\alpha (1)\right)\le [{\alpha }^{-1}H\alpha ]$, where $[\alpha^{-1} H \alpha]=\lbrace [\alpha^{-1}\gamma \alpha] \ \vert \ [\gamma]\in{H}\rbrace$.\\
\textbf{(iii)} $X$ is called {\textit{semilocally }}${\mathbf H}${\textit{-connected}} if for every $x\in{X}$ and for every path $\alpha $ from $x_{0}$ to $x$, the space $X$ is semilocally $[{\alpha }^{-1}H\alpha ]$-connected at $x\in{X}$.
\end{definition}

Note that in the case of trivial subgroup $H=e$ both of semilocally path $H$-connected and semilocally $H$-connected properties coincide with the semi locally simply connected property.
\begin{theorem}
Let $H\le {\pi }_1(X,x_0)$, then $X$ is semilocally H-connected at ${{\rm x}}_0\in {\rm X}$ if and only if $H$ is an open subgroup of ${{\pi }_1}^{wh}(X,x_0)$.
\end{theorem}
\begin{proof}
Let $X$ be semilocally H-connected at $x_{0}\in{ X}$, i.e, there exists an open neighborhood $U$ $of\ x_0$ such that  $i_*{\pi }_1(U,x_0)\le H$. Thus for every $[\alpha ]\in H$, $[\alpha]i_*{\pi }_1(U,x_0)\le H$. By Lemma 3.1, $H$ is an open subgroup of ${{\pi }_1}^{wh}(X,x_0)$. Conversely, let $[\alpha ]\in H$. Since $H$ is an open subgroup of ${{\pi }_1}^{wh}(X,{x}_0)$ then there exists an open neighborhood $U$ $of\ x_0$ such that  ${[\alpha ]i}_*{\pi }_1(U,x_0)\le H$. Therefore $i_*{\pi }_1(U,x_0)\le [{\alpha }^{-1}]H=H$. Hence $X$ is semilocally H-connected at ${{\rm x}}_0\in {\rm X}$.
\end{proof}

Naturally, semilocally H-connectedness at a point depends on the choice of the point. For instance, the Hawaiian earring $(HE,a)$ is semilocally simply connected at $x\in HE$, where $x$ is any non-based point, $x\ne a$, and clearly it is not semilocally simply connected at $a\in HE$.

We recall from \cite{16} that the Spanier group $\pi(\mathcal{U}, x_{0})$ with respect to an open cover $\mathcal{U}=\lbrace U_{i} \ \vert \ i\in{I}\rbrace$ is defined to be the subgroup of $\pi_{1}(X,x_{0})$ which contains all homotopy classes having representatives of the following type:
$$\prod_{j=1}^{n}\alpha_{j}\beta_{j}\alpha^{-1}_{j},$$
where $\alpha_{j}$'s are arbitrary path starting at $x_{0}$ and each $\beta_{j}$ is a loop inside of the open set $U_{j}\in{\mathcal{U}}$.

Also, we recall from  \cite[Definition 3.1]{16} that the path Spanier group $\tilde{\pi}(\mathcal{V}, x_{0})$ with respect to a path open cover $\mathcal{V}=\lbrace V_{\alpha} \ \vert \ \alpha\in{P(X,x_{0})}\rbrace$ of $X$ such that $\alpha(1)\in{V_{\alpha}}$ for every $\alpha\in{P(X,x_{0})}$, is defined to be the subgroup of $\pi_{1}(X,x_{0})$  which contains all homotopy classes  having representatives of the following type:
$$\prod_{j=1}^{n}\alpha_{j}\beta_{j}\alpha^{-1}_{j},$$
where $\alpha_{j}$'s are arbitrary path starting at $x_{0}$ and each $\beta_{j}$ is a loop inside of the open set $V_{\alpha_{j}}$ for all $j\in{\lbrace1,2,...,n\rbrace}$.

Let $X$ be semilocally path $H$-connected, then by the definition, for every path $\alpha\in p(X,x_{0})$, there exists an open neighborhood $V_{\alpha}$ which for every loop $\gamma :(I,\dot{I})\rightarrow (V_{\alpha},\alpha(1))$, $[\gamma]\in{[\alpha^{-1}H\alpha]} $ if and only if $ [\alpha^{-1}\gamma\alpha]\in{H}$. Put $\mathcal{V}=\lbrace V_{\alpha} \ \vert \ \alpha\in{P(X,x_{0})}\rbrace$ and construct $\tilde{\pi}(\mathcal{V},x_{0})$ such as the above. It implies easily that  $\tilde{\pi}(\mathcal{V},x_{0})\leq H$. On the other hand, Torabi et al. \cite[Theorem 4.1]{16}  showed that $H$ is an open subgroup of ${{\pi }_1}^{qtop}(X,x_0)$ if and only if there exists a path open cover $\mathcal{V}$ of $x_0$ such that $\widetilde{\pi }(\mathcal{V},x_0)\le H$. The result can be seen in the following proposition.
\begin{proposition}
A connected, locally path connected space $X$ is semilocally path H-connected for $H\le {\pi }_1(X,x_0)$ if and only if $H$ is an open subgroup of ${{\pi }_1}^{qtop}(X,x_0)$.
\end{proposition}
\begin{corollary}
For a connected, locally path connected space $X$, the categorical equality $\mathrm{SCov}(X)=\mathrm{GCov}(X)$ holds if and only if $X$ is semilocally path ${{\pi }_1}^{gc}(X,x_0)$-connected.
\end{corollary}
\begin{proof}
Assume that the categorical equality $\mathrm{SCov}(X)=\mathrm{GCov}(X)$ holds, then every generalized covering subgroup and so ${{\pi }_1}^{gc}(X,x_0)$ is a semicovering subgroup. By \cite[Theorem 4.1]{16} there is a path open cover $\mathcal{V}$ of $X$ such that $\widetilde{\pi }(\mathcal{V},x_0)\le {{\pi }_1}^{gc}(X,x_0)$. Using Proposition 4.3, since $\widetilde{\pi }(\mathcal{V},x_0)$ is open, $X$ is semilocally path ${{\pi }_1}^{gc}(X,x_0)$-connected. By a similar argument, one can prove the converse.
\end{proof}
\begin{corollary}
If $X$ is homotopicaly Hausdorff relative to $H$ and the index of $H\ in\ {\pi }_1(X,x_0)$ is finite, then $X$ is semilocally path \textit{H}-connected.
\end{corollary}
\begin{proof}
Since the index of $H$  in $\pi_{1}(X,x_{0})$ is finite, by Proposition 2.13, $H$ is an open subgroup of $\pi_{1}^{qtop}(X,x_{0})$. The result follows by Proposition 4.3.
\end{proof}

The following proposition states  same result for semilocally ${\rm H}$-connected spaces.
\begin{proposition}
 A connected, locally path connected space $X$ is semilocally ${\rm H}$-connected for $H\le {\pi }_1(X,x_0)$ if and only if $H$ is a covering subgroup of ${\pi }_1(X,x_0)$.
\end{proposition}
\begin{proof}
Remember that  $H$ is a covering subgroup of ${\pi }_1(X,x_0)$ if and only if there is an open cover $\mathcal{U}$ of $X$ such that the Spanier subgroup $\pi (\mathcal{U},x_0)$ of ${\pi }_1(X,x_0)$ is a subgroup of $H$ (see \cite{21}). Let $H\leq \pi_{1}(X,x_{0})$ and assume that $X$ is semilocally $H$-connected space, then for every $x\in{X}$ and for any path $\alpha:x_{0}\mapsto x$, there is an open neighborhood $U$ at $x$ such that $i_{\ast}\pi_{1}(U,x)\leq [\alpha^{-1}H \alpha]$, equivalently, $[\alpha]i_{\ast}\pi_{1}(U,x)[\alpha^{-1}]\leq H$. Put  $\mathcal{U}=\lbrace U \ \vert \ x\in{X}, \ \alpha:x_{0}\mapsto x \rbrace$. It is clear that $\mathcal{U}$ is an open cover of $X$ and $\pi(\mathcal{U}, x_{0})\leq H$. Similarly, one can prove the converse.
\end{proof}
\begin{corollary}
For a connected, locally path connected space $X$, the categorical equality $\mathrm{GCov}(X)=\mathrm{Cov}(X)$ holds if and only if $X$ is semilocally ${{\pi }_1}^{gc}(X,x_0)$-connected.
\end{corollary}
\begin{proof}
Assume that $X$ is semilocally ${{\pi }_1}^{gc}(X,x_0)$-connected. Using Proposition 4.6, ${{\pi }_1}^{gc}(X,x_0)$ is a covering subgroup of ${\pi }_1(X,x_0)$. Let $H$ be a generalized covering subgroup of ${\pi }_1(X,x_0)$, then clearly ${{\pi }_1}^{gc}(X,x_0)\le H$. By \cite[Proposition 3.3]{15}, $H$ is a covering subgroup. Therefore $\mathrm{GCov}(X)\le \mathrm{Cov}(X)$. The equality holds from  Remark 2.15. Conversely, if $\mathrm{GCov}(X)=\mathrm{Cov}(X)$, then clearly ${{\pi }_1}^{gc}(X,x_0)$ is a covering subgroup of ${\pi }_1(X,x_0)$ and so by Proposition 4.6 $X$ is semilocally ${{\pi }_1}^{gc}(X,x_0)$-connected.
\end{proof}
\begin{corollary}
 If $X$ is a connected, locally path connected and semi locally ${{\pi }_1}^{gc}(X,x_0)$-connected space, then ${{\pi }_1}^{gc}(X,x_0)={{\pi }_1}^{sp}(X,x_0)$ and $X$ is coverable in the sense of \cite{15}.
\end{corollary}
\begin{corollary} Every semilocally $H$-connected space  is semilocally path H-connected.
\end{corollary}











\end{document}